\documentclass[a4paper,10pt]{article}
\usepackage{graphicx,fancyhdr,array,multicol,multirow,enumitem}
\usepackage[margin=1.5cm,vmargin={1.5cm,1.5cm},includefoot]{geometry}
\usepackage{fancybox}
\usepackage{xcolor}
\usepackage{cancel}
\usepackage[normalem]{ulem}
\usepackage{amsthm}
\usepackage[utf8]{inputenc}
\usepackage[T1]{fontenc}
\usepackage[english]{babel}
\usepackage{amssymb,mathtools,amsfonts}
\usepackage{stmaryrd}

\usepackage{cite} 
\newtheorem{lemma}{Lemma}
\newtheorem{theorem}{Theorem}

\newtheorem{remark}{Remark}

\newcommand{\eps}{\varepsilon}

\newcommand{\bL}{\mathbb L}
\newcommand{\N}{\mathbb N}
\renewcommand{\P}{\mathbb P}
\newcommand{\R}{\mathbb R}

\newcommand{\Z}{\mathbb Z}

\title{A note on the antisymmetry in the speed of a random walk in reversible dynamic random environment}
\author{Oriane Blondel}
\begin{document}

\maketitle
\begin{abstract}
In this short note, we prove that $v(-\eps)=-v(\eps)$. Here, $v(\eps)$ is the speed of a one-dimensional random walk in a dynamic \emph{reversible} random environment, that jumps to the right (resp. to the left) with probability $1/2+\eps$ (resp. $1/2-\eps$) if it stands on an occupied site, and vice-versa on an empty site. We work in any setting where $v(\eps), v(-\eps)$ are well-defined, \emph{i.e.\@} a weak LLN holds. The proof relies on a simple coupling argument that holds only in the discrete setting. 
\end{abstract}
\section{Introduction}

We consider the so-called ``$\eps$--random walk'': a random walk in one-dimensional dynamic random environment with two values that jumps to the right (resp. to the left) with probability $1/2+\eps$ (resp. $1/2-\eps$) if it stands on an occupied site, and vice-versa on an empty site (Figure \ref{fig:rw}).

\begin{figure}[h]
\label{fig:rw}
\begin{center}
\includegraphics[scale=.5]{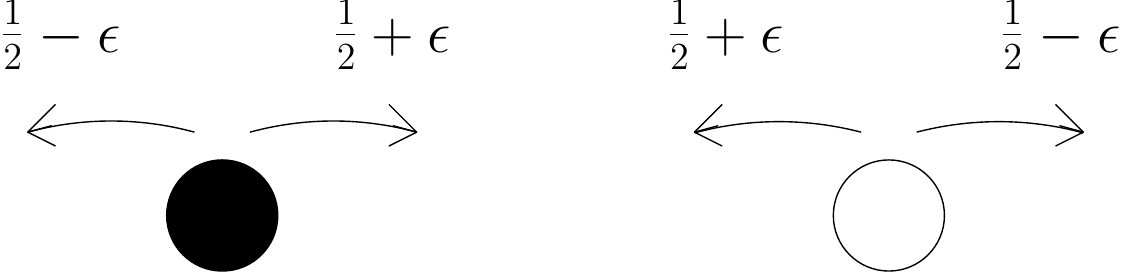}
\caption{Jump rates for the $\eps$--random walk.}
\end{center}
\end{figure}

A lot of energy has been devoted to describing the behavior of this $\eps$--RW in various random environments, mainly to find whether it satisfies the usual limit theorems (LLN, CLT...). This is generally a hard problem, since the environment seen from the $\eps$--RW is highly non-stationary (even when the environment is). Additionally, it belongs to the class of nestling random walks, which do not have an a-priori preferred direction. It is possible to find settings in which the LLN does not hold. \cite[Section 9]{cavalinhos} proposes an example in which space-time traps lead the random walk into longer and longer stretches of motion with drifts to the right or the left. Cases in which the LLN has been shown include perturbative regimes (small $\eps$) \cite{ABF18}, well-known environments like the exclusion process \cite{HS15,HKT20} or the contact process \cite{MV15}, uniform mixing hypotheses \cite{CZ04,AdHR11,RV13}, or fast enough decay of correlations \cite{cavalinhos}.

When the random walk does satisfy a law of large numbers, let us call $v(\eps)$ its asymptotic speed (we will simply say that $v(\eps)$ is well-defined). 
We observed in \cite{ABF16} that, if the environment is given by a reversible Markov process with positive spectral gap, and $|\eps|$ is smaller than the spectral gap, $v$ --in addition to being well-defined-- satisfies the antisymmetry property
\begin{equation}\label{antisym}
v(-\eps)=-v(\eps).
\end{equation}
This property also holds in higher dimensions for random walks with a certain symmetry property (Assumption 3 in \cite{ABF16}). Simulations moreover suggest that this property extends out of the perturbative regime of \cite{ABF16} (where the law of large numbers was not yet even established; see Figure 3 in \cite{ABF16}). The proof relies on an expansion of the speed in $\eps$, in which the terms of even degrees are shown to cancel due to reversibility.

It is worth pointing out that this antisymmetry does not seem to follow from obvious symmetry properties of the system. Indeed, it is tempting at first sight to think that the $(-\eps)$--random walk reversed in time should have the same distribution as the $\eps$--random walk, from which \eqref{antisym} would immediately follow. But closer inspection reveals that this is not true and those two processes have very different trajectories in general (see for instance Figure 4 in \cite{ABF16}). Rather, \eqref{antisym} holds iff the speed of the $\eps$--random walk is the same in the reversible environment and in its image under the mirror symmetry $x\mapsto -x$. While this property is obvious for reversible systems invariant under mirror symmetry (e.g.\@ simple symmetric exclusion process), there exist reversible, translation invariant processes which have no mirror-invariance. One that has been extensively studied is the East model \cite{eastphys,eastrecent} , in which particles appear (resp.\@ disappear) at site $x$ at rate $p$ (resp.\@ $1-p$), but only if $x+1$ is empty. This dynamics is reversible w.r.t.\@ the product Bernoulli measure of density $p$ (it is easy to check detailed balance since the constraint required to update $\eta(x)$ does not involve $\eta(x)$), and clearly non mirror-symmetric. In fact, any kinetically constrained model with non-symmetric constraint would yield an example where the result below is non-empty.

To state our main result, let $(X^\eps_t)_{t\geq 0}$ be the trajectory of a $\eps$--RW started from $0$. A proper construction of this object is given in Section~\ref{s:setting}.

\begin{theorem}\label{t:main}
	Assume the environment is a translation invariant and reversible Markov process. Moreover, assume that $X^\eps$ and $X^{-\eps}$ satisfy a weak law of large numbers, \emph{i.e.\@} there exist $v(\eps),v(-\eps)$ such that
	
	\begin{align}\label{e:LLN}
	\frac{1}{t}X^{\pm \eps}_t\underset{t\rightarrow \infty}{\overset{\P}{\longrightarrow}}v(\pm \eps).
	\end{align} Then
	\begin{equation}
	v(-\eps)=-v(\eps).
	\end{equation}
	\end{theorem}
Checking that the weak law of large numbers holds is generally not easy. In \cite{cavalinhos}, it was established under the condition that the environment has polynomially decaying covariances with a high enough exponent. This includes Markov processes with positive spectral gap. The East model \cite{eastphys,eastrecent} is an example of such a model with no mirror symmetry. 

One point that the theorem above does not address is whether the limiting speed is non-zero. This is in general hard to settle. \cite{HS15} gives a nice example in perturbative settings where the update rates of the environment are either very large or very small. \cite{cavalinhos} also offers a sufficient criterion for non-zero speed, albeit rather limited in its application range. Theorem~\ref{t:main} does say that $v(\eps)$ is non-zero whenever $v(-\eps)$ is, which is in itself not obvious.

Let us note that, to the statistical physics community, the hidden symmetry of Theorem~\ref{t:main} in a reversible context may be reminiscent of Gallavotti-Cohen type results. However, it is unclear what the connection is, if it exists at all. Indeed, the action functional of the environment seen from the particle does not seem to give information about the displacement of the particle.

The proof presented here relies on a coupling valid only for the discrete time $\eps$--RW, which in turn allows to deduce the result for the continuous time. The coupling relies heavily on the fact that we consider nearest-neighbor trajectories in dimension $1$. However, as mentioned above, Theorem~\ref{t:main} is expected to hold also in higher dimensions and for random walks allowing longer range jumps under appropriate symmetry assumptions \cite{ABF16}. Outside the perturbative region, this is still an open problem.
%
%
%
\section{Setting and construction of the random walk}\label{s:setting}
We now give explicit constructions of the $\eps$--RW, in discrete and continuous time. In the following, $\bL=\R_+$ when we consider the continuous time setting, $\bL=\N$ in the discrete time setting.

\subsection{Environment}
The environment is given by a collection $\eta=(\eta_t(x))_{(x,t)\in\Z\times\bL}$ such that $\eta_t(x)\in\{0,1\}$ for all $(x,t)\in\Z\times\bL$. We assume 
  \begin{enumerate}
	\item \textbf{(Stationarity)} The distribution of the environment is invariant under space-time shifts, \emph{i.e.} 
	\begin{equation}
	\forall x\in\Z,t\in\bL,\quad (\eta_{t+s}(x+\cdot))_{s\in\bL}=(\eta_s(\cdot))_{s\in\bL}\quad \text{in distribution.}
	\end{equation}
	\item \textbf{(Reversibility)} The environment is given by a reversible process $\eta=(\eta_t)_{t\in\bL}$ on $\{0,1\}^{\Z}$, \emph{i.e.} 
	\begin{equation}\text{for all $T\in\bL\setminus\{0\}$,}\quad (\eta_{T-t})_{t\in[0,T]\cap \bL}=(\eta_t)_{t\in[0,T]\cap\bL}\quad \text{in distribution.}
	\end{equation}
	\end{enumerate}

\subsection{$\eps$--RW in discrete time}\label{s:discrete}

Let us consider a collection $U=(U_{x,n})_{x\in\Z,n\in\N}$ of iid $\mathcal{U}([0,1])$ random variables. With the environment $\eta$ and the collection $U$ we associate a set $(A^\eps_{x,n})_{x\in\Z,n\in\N}$ (that will prescribe the directions taken by the RW) in the following way:
\begin{equation}
A^\eps_{x,n}=\begin{cases}
+1& \text{if }\big(\eta_n(x)=1 \text{ and }U_{x,n}\leq \frac{1}{2}+\eps \big)\text{ or }\big(\eta_n(x)=0 \text{ and }U_{x,n}\leq \frac{1}{2}-\eps \big),\\
-1&\text{else}.
\end{cases}
\end{equation}

For $x\in\Z$, we denote by $(X^{x,\eps}_n)_{n\in\N}$ the $\eps$--RW started at $x$. It is built iteratively as follows:
\begin{enumerate}
	\item $X^{x,\eps}_0=x$;
	\item if $X^{x,\eps}_n=y$, $X^{x,\eps}_{n+1}=y+A^\eps_{x,n}$.
	\end{enumerate}

Note that by translation invariance of the environment, $X^{x,\eps}-x\overset{(d)}{=}X^{0,\eps}$ for all $x\in\Z$. Also, $X_n^{x,\eps}$ has the same parity as $n+x$. 
\subsection{$\eps$--RW in continuous time}\label{s:cont}

%
%

We could use a similar construction, but it will be more convenient to alter it a little. Let $U=(U_n)_{n\in\N}$ be a collection of iid $\mathcal{U}([0,1])$ random variables and $T=(T_n)_{n\in\N}$ a PPP($1$) on $\R_+$. Assume $(U,T,\eta)$ to be independent. The continuous time $\eps$--RW started at $x$, $(X^{x,\eps}_t)_{t\in\R_+}$  is built as follows:
\begin{enumerate}
	\item $X^{x,\eps}_0=x$;
	\item if $X^{x,\eps}_{t^-}=y$ and $y\notin\{T_{n},n\in\N\}$, $X^{x,\eps}_{t}=y$;
	\item if $X^{x,\eps}_{t^-}=y$ and $t=T_{n}$,  \begin{equation}X^{x,\eps}_{t}=\begin{cases}y+1 &\text{if }\big(\eta_t(y)=1 \text{ and }U_{n}\leq \frac{1}{2}+\eps \big)\text{ or }\big(\eta_t(y)=0 \text{ and }U_{n}\leq \frac{1}{2}-\eps \big),\\
	y-1&\text{else}.\end{cases}\end{equation}
\end{enumerate}

\section{Proof of Theorem~\ref{t:main} in the discrete setting}

A key remark is that the reversibility assumption allows to construct jointly a $\eps$--RW and a backwards $(-\eps)$--RW, that is a $(-\eps)$--RW evolving on the reversed environment process $(\eta_{N-n})_{n\leq N}$ for $N\in\N$.

In order to simplify the notations, let us fix $\eps\in [-1,1]$, and $N\in\N$. We let $X=X^{0,\eps}$, and for $x\in\Z$ we let $(\widehat X_n)_{n\leq N}$ be constructed as follows.
\begin{enumerate}
	\item $\widehat X_N=x$;
	\item for $n<N$, if $\widehat X_{N-n}=y$, $\widehat X_{N-n-1}=y-A^\eps_{y,N-n}$.
	\end{enumerate}
Note that the construction uses the same collection $A^\eps$ as the construction of $X$.

Reversibility of the environment clearly gives us the following property.
\begin{lemma}\label{rem}
	$(\widehat X_n)_{n\leq N}\overset{(d)}{=}(X^{x,-\eps}_n)_{n\leq N}$.
	\end{lemma}

\begin{proof}
	This follows from the reversibility of $\eta$.
	\end{proof}
It is also not difficult to check that $\widehat X$ and $X$ cannot cross, in the following sense.

\begin{lemma}\label{l:non-crossing}
	Let $X$ and $\widehat X$ be built using the same collection $A^\eps$. Then, if $\widehat X_0$ and $X_0$ have the same parity, 
	\begin{equation}
	(\widehat X_0-X_0)(\widehat X_N-X_N)\geq 0.
	\end{equation}
	\end{lemma}

\begin{proof}
Without loss of generality, let us assume $\widehat X_N>X_N$. Assume by contradiction that $X_0>\widehat X_0$. With our parity assumption, for any $n\leq N$, $\widehat X_n-X_n\in 2\Z$, and that difference only takes steps in $\{-2,0,2\}$. Let us look at the largest $n<N$ such that $\widehat X_n=X_n$. Let $y=X_n$. Under our assumptions, necessarily $\widehat X_{n+1}=y+1$ and $X_{n+1}=y-1$. In particular, $A^\eps_{y,n}=-1$, and therefore $\widehat X_{n-1}=y+1\geq X_{n-1}$. Now, if $\widehat X_{n-1}>X_{n-1}$, we reproduce the previous argument to show $\widehat X_{n-2}\geq X_{n-2}$ (unless $n-1$ is zero, in which case we stop). If $\widehat X_{n-1}=X_{n-1}=y+1$, $A^\eps_{y+1,n-1}=-1$, and we also end up with $\widehat X_{n-2}=y+2\geq X_{n-2}$. We can reproduce this argument until we reach time $0$ to get a contradiction.
	\end{proof}

With this non-crossing property, it becomes easy to show the antisymmetry property. 

\begin{proof}[Proof of Theorem~\ref{t:main} (discrete setting)]
	Without loss of generality, assume by contradiction that $\delta:=v(\eps)+v(-\eps)>0$. Fix $N$ large enough that 
	\begin{align}
	\P\left(\left|\frac{X_N^{\pm \eps}}{N}-v(\pm \eps)\right|\geq \delta/4\right)\leq 1/3.
	\end{align}

Define $\tilde x=\lfloor(v(\eps)-\delta/2)N\rfloor$. If $\tilde x$ has the same parity as $N$, let $x=\tilde x$. Let $x=\tilde x +1$ else. Define $X,\widehat X$ as above. Consider the event 
\begin{equation}
E=\left\{\frac{X_N}{N}\geq v(\eps)-\delta/4\text{ and }\frac{\widehat X_0}{N}-x\geq v(-\eps)-\delta/4\right\}.
\end{equation}

By Lemma~\ref{rem}, $\P(E)\geq 1/3$, but by Lemma~\ref{l:non-crossing}, $E$ is empty, since $X_N>\widehat X_N=x$ and $X_0<\widehat X_0$ on $E$. We have our contradiction.

	\end{proof}

\begin{remark}
	While the non-crossing property is elementary to check in the discrete time setting, it is not easy to find a coupling with the same property in the continuous time setting. One could think of associating independent Poisson clocks to every vertex of $\Z$, a direction $A^\eps_{x,n}$ with every clock ring and build $X,\widehat X$ in a similar fashion. However, the jumps starting from neighboring sites would not be simultaneous, and it is easy to find examples where crossings occur. Fortunately, once we have the result in the discrete time setting, we do not need to work hard to push it to the continuous time setting.
	\end{remark}
\section{Proof of Theorem~\ref{t:main} in the continuous setting}

Recall the construction of Section~\ref{s:cont}. Define for $n\in\N$, $\sigma_n=\eta_{T_n}$. Theorem~\ref{t:main} will be proved when we have checked the three points of the following lemma.

\begin{lemma}\label{l:cont}\begin{enumerate}
		\item The process $\sigma$ is translation invariant and reversible.
	\item If $X^\eps$ is the $\eps$--RW in continuous time defined on $\eta$ as in Section~\ref{s:cont}, $(X^{\eps}_{T_n})_{n\in\N}$ is equal in distribution to $\widetilde X^\eps$ the $\eps$--RW in discrete time defined on $\sigma$.
	\item $X^\eps$ satisfies the weak LLN \eqref{e:LLN} iff $\widetilde X^\eps$ does, with the same limit speed $v(\eps)$.
	\end{enumerate}
	\end{lemma}

\begin{proof}[Proof of Lemma~\ref{l:cont}]
	\begin{enumerate}
		\item This follows from the translation invariance and reversibility of $\eta$ and the PPP $T$, along with the independence of $T$ and $\eta$.
		\item This follows from the independence of $U,T,\eta$, and a construction in the discrete case analog to the one we use in Section~\ref{s:cont}.
		\item This is an immediate consequence of the fact that $T_n/n\underset{n\rightarrow\infty}{\longrightarrow}1$ almost surely.
		\end{enumerate}
	\end{proof}

\begin{proof}[Proof of Theorem~\ref{t:main} (continuous setting)]
	We assume that both $X^\eps$ and $X^{-\eps}$ satisfy the weak LLN \eqref{e:LLN}. By the third point of Lemma~\ref{l:cont}, $\widetilde X^\eps$ and $\widetilde X^{-\eps}$ satisfy \eqref{e:LLN} with respective speeds $v(\eps), v(-\eps)$. Therefore, the discrete time version of Theorem~\ref{t:main} applies to $\widetilde X^\eps$, $\widetilde X^{-\eps}$ and $v(-\eps)=-v(\eps)$.

\end{proof}

\bibliographystyle{plain}
\bibliography{biblio-antisym}

\end{document}